\documentclass[12pt]{amsart}

\usepackage{setspace,amsfonts,amsthm,amssymb,amsmath,comment,fancyhdr,commath}
\usepackage[mathscr]{euscript}
\usepackage[utf8]{inputenc}
\usepackage[T1]{fontenc}

\newtheorem{theorem}{Theorem}[section]

\newtheorem{prop}[theorem]{Proposition}

\theoremstyle{definition}

\theoremstyle{remark}

\newcommand\mR{{\mathbb R}}

\newcommand\mN{{\mathbb N}}

\newcommand\restr{\mathord{\restriction}}

\newcommand\B{{\mathscr B}}

\newcommand\Cia{{\mathscr C}}
\newcommand\cn{\colon}
\newcommand\om{\Omega}

\newcommand\mn{\leqslant}

\newcommand\ol{\overline}
\newcommand\al{\alpha}

\newcommand\fs{F_\sigma}
\newcommand\gd{G_\delta}

\numberwithin{equation}{section}

\begin{document}

\title[Kuratowski Extension Theorem for Baire-alpha functions]{An approximation form of the Kuratowski Extension Theorem for Baire-alpha functions}

\author{Waldemar Sieg}
\curraddr{Institute of Mathematics, Kazimierz Wielki University,
Powsta\'nc\'ow Wielkopolskich 2,
85-090 Bydgoszcz, Poland}
\email{waldeks@ukw.edu.pl}


\date{}


\begin{abstract}
Let $\om$ be a perfectly normal topological space, let $A$ be a non-empty $\gd$-subset of $\om$ and let $\B_1(A)$ denote the space of all functions $A\to\mR$ of Baire-one class on $A$.
Let also $\|\cdot\|_\infty$ be the supremum norm. The symbol $\chi_A$ stands for the characteristic function of $A$. We prove that for every bounded function $f\in\B_1(A)$ there is a sequence $(H_n)$
of both $\fs$- and $\gd$-subsets of $\om$ such that the function $\ol{f}\cn\om\to\mR$ given by the uniformly convergent series on $\om$ with the formula:
$
\ol{f}:=c\sum_{n=0}^\infty\left(\frac{2}{3}\right)^{n+1}\left(\frac{1}{2}-\chi_{H_n}\right)
$
extends $f$ with $\ol{f}\in\B_1(\om)$, $c=\sup_{x\in\om}\abs{\ol{f}(x)}$ and the condition ($\triangle$) of the form:
$\|f\|_\infty=\|\ol{f}\|_\infty$.
We apply the above series to obtain an extension of $f$ positive to $\ol{f}$ positive with the condition ($\triangle$). A similar technique allows us to obtain an extension of Baire-alpha function
on $A$ to Baire-alpha function on $\om$.
\end{abstract}

\maketitle

\section{Notation}
Let $\om$, $A$, $\chi_A$, $\|\cdot\|_\infty$ and $\B_1(A)$ have the same meanings as in the abstract.
If $f\in\B_1(A)$ and $B\subset A$ then $f_{\restr B}$ denotes the restriction
of $f$ to $B$. Recall that $\fs(\om)$ ($\gd(\om)$, respectively) is the family of countable
sums of closed subsets (countable intersections of open subsets, respectively) of $\om$.
Note that $\B_0(A)=\Cia(A)$, the space of all continuous functions on $A$. $\B_\al(A)$ is the space of pointwise limits of sequences in $\bigcup_{\xi<\al}\B_\xi(A)$ (for more details in that field see \cite{Jay1,Jay2}).
The symbol $\B_\al^{bd}(A)$ stands for the family of bounded elements of $\B_\al(A)$.
The following theorem is a special case of a more general result (it will be given in Section 4 of this paper) which is crucial for our future considerations. It gives a characterization of Baire-one functions in terms of $\fs$-sets.

\begin{prop}(Lebesgue-Hausdorff, \cite[p.393]{Kuratowski2})\label{prop1}
Let $\om$ be a perfectly normal topological space and let a function $f\cn\om\to\mR$ be bounded.
Then $f\in\B_1(\om)$ if and only if $f^{-1}(F)\in\gd(\om)$ for every closed subset $F$ of $\mR$.
\end{prop}

\section{Introduction}
Let $\al$ be an ordinal number such that $0\mn\al<\omega_1$.
We consider the following \emph{Extension Problem}:
\begin{center}
Given $\al$, $A$ and $\om$ as above, does every element $f\in\B_\al(A)$
have an extension $\ol{f}\in\B_\al(\om)$?
\end{center}
So far, all known solutions to the Extension Problem deal with the case $\al$ finite
(of course for $\om$ normal, $A$ closed and $\al=0$, the Tietze theorem is a solution).
The first result in this direction was proved in $1933$ by Kuratowski
\cite[p.434]{Kuratowski1}:
\begin{center}
(K) If $\om$ is a metric space and $A$ is a $\gd$ subset of $\om$ then every element $f\in\B_1(A)$
has an extension $\ol{f}\in\B_1(\om)$.
\end{center}
In the last twenty years a few refinements and generalizations of Kuratowski's theorem
have been proven. In $2003$ Leung and Tang showed \newline\cite[Theorem 3.6]{LeungTang1} that there are extensions $\ol{f}$ preserving the so-called oscillation index of $f$,
whene\-ver $A$ in (K) is closed (cf. \cite{LeungTang2}). On the other hand, in $2005$ Kalenda and Spurn\'y
\cite[Theorem 10(a)]{KalendaSpurny} strengthened (K) with the condition
\begin{description}
\item[$(*)$]
\begin{center}
   $\sup{f(A)}=\sup\ol{f}(\om)$ and $\inf f(A)=\inf\ol{f}(\om)$
\end{center}
\end{description}
and showed \cite[Example 19]{KalendaSpurny} that if in (K) the word ''$\gd$-subset''
is replaced by ''$\fs$-subset'', the implication becomes false, in general
(e.g., for $A=$ the rational numbers of $X=[0,1]$ endowed with the natural metric).
In $2004$ Shatery and Zafarani \cite[Theorem 1.7]{ShatZaf}, using a number of technical
lemmas, obtained a generalization of the Kuratowski theorem (K). They proved that
\begin{center}
\emph{The Extension Problem has a positive answer for $\al$ finite, $\om$ perfectly normal
and a Borel subset $A\subset\om$ of multiplicative class $\al$ (in particular, for $A\in\gd(\om)$).}
\end{center}

During our research we were looking for a brief generalization of the Kuratowski Extension Theorem. It turned out that this can be done by using of the Uryshon's method of proof and the Lebesgue-Hausdorff Theorem with the assumption that $\om$ is perfectly normal. The investigation gave us an explicit extension formula for every bounded function $f\in\B_\al(A)$. For simplicity, in Section 3 we present a full proof for the case $\al=1$ only
(Theorem \ref{MR}). The justification for $1<\al<\omega_1$ is analogous (we write about it in Section 4).

\section{Extensions of Baire-one functions}\label{EXT}
The proof of the main theorem of this section is based on two auxiliary results. First of them
is (the Lebesgue-Hausdorff) Proposition \ref{prop1}. The second one is the following special case of the well-known Separation Theorem. It allows to separate two disjoint $\gd$-subsets of a perfectly normal topological space $\om$ by the set which is $\fs$ and $\gd$ in $\om$. In the cited Srivastava's book \cite{Srivastava} the result is formulated for a metrizable space, but simple transformations allows us to state that it is also true for a domain that is perfectly normal.

\begin{prop}(\cite[Theorem 3.6.11]{Srivastava})\label{prop2}
Let $M,N$ be two (non-empty) disjoint $\gd$-subsets of a perfectly normal topological space $\om$. Then there exists a set $H$ of both $\fs$- and $\gd$-subset of $\om$ such that $M\subset H$ and $N\cap H=\emptyset$.
\end{prop}

We are now ready to formulate our first main result.

\begin{theorem}\label{MR}
Let $\om$ be a perfectly normal topological space and let $A$ be a $\gd$-subset of $\om$. Then every function $f\in\B_1(A)$ has an extension $\ol{f}\in\B_1(\om)$ such that
$\|f\|_\infty=\|\ol{f}\|_\infty$.

More exactly, if $\sup_{x\in A}\abs{f(x)}=c<\infty$, then there exists a sequence
$\left(H_n\right)$ of both $\fs$- and $\gd$-subsets of $\om$, such that the function
$$
\ol{f}(x):=c\cdot\sum_{n=0}^\infty\left(\frac{2}{3}\right)^{n+1}\left(\frac{1}{2}-\chi_{H_n}(x)\right)
~\mathrm{(uniform~convergence~on\ \om)}
$$
extends $f$ with $\sup_{x\in\om}\abs{\ol{f}(x)}=c$.
\end{theorem}
\begin{proof}
As in the Tietze theorem, it is enough to consider the case $c=1$.
Put $M=f^{-1}\left(\left[-1,-\frac{1}{3}\right]\right)$ and
$N=f^{-1}\left(\left[\frac{1}{3},1\right]\right)$. By Proposition \ref{prop1},
both the (disjoint) sets $M$ and $N$ are $\gd$-subsets of $A$. Since $\gd(A)=A\cap\gd(\om)$
and $A\in\gd(\om)$, we have $M,N\in\gd(\om)$. By Proposition \ref{prop2},
there is a set $H_0$ of both $\fs$- and $\gd$-subset of $\om$
such that $M\subset H_0$ and $N\cap H_0=\emptyset$. By Proposition \ref{prop1} again, the characteristic function $\chi_{H_0}$ belongs to $\B_1(\om)$
(cf. \cite[Lemma 1.4]{ShatZaf}).
Put $g_0(x)=\frac{2}{3}\left(\frac{1}{2}-\chi_{H_0}(x)\right)$. Hence $g_0\in\B_1(\om)$,
\begin{equation}\label{1}
\abs{g_0(x)}=\frac{1}{3}\textrm{ for all }x\in\om\textrm{ and}
\end{equation}
\begin{equation}\label{2}
\abs{f(x)-g_0(x)}\mn\frac{2}{3}\textrm{ for all }x\in A.
\end{equation}
Now we mimic the classic proof of the Tietze theorem \cite[p.69]{Eng}.
Repeating the above arguments we obtain a sequence $\left(H_n\right)_{n=0}^\infty$
 of both $\fs$- and $\gd$-subsets of $\om$ such that, for the functions
$g_n\in\B_1(\om)$, $n=0,1,\ldots$, of the form
$g_n(x):=\left(\frac{2}{3}\right)^{n+1}\left(\frac{1}{2}-\chi_{H_n}(x)\right)$
we have
\begin{equation}\label{1'}
\abs{g_n(x)}=\frac{1}{2}\left(\frac{2}{3}\right)^{n+1}\textrm{ for all }x\in\om\textrm{ and}
\end{equation}
\begin{equation}\label{2'}
\abs{f(x)-(g_0(x)+\ldots+g_n(x))}\mn\left(\frac{2}{3}\right)^{n+1}\textrm{ for all }x\in A.
\end{equation}
By (\ref{1'}) and by the fact that the space $\B_1(\om)$ is closed with respect to uniform limits
(see \cite[p.269]{Hausdorff}), the series $\sum_{n=0}^\infty g_n$
converges uniformly on $\om$ to a function $g\in\B_1(\om)$. By (\ref{2'}) we obtain
$g_{\restr A}=f$ and, by (\ref{1'}) again, $\sup_{x\in\om}\abs{g(x)}=1$.
Hence $\ol{f}:=g$ is the required extension of $f$.
\end{proof}

Note that Theorem \ref{MR} has a version for non-negative mappings.
\begin{theorem}\label{PMR}
Let $\om$ be a perfectly normal topological space and let $A$ be a $\gd$-subset of $\om$. Moreover,
let $0\mn f\in B_1^{bd}(A)$ with $\sup_{x\in A}f(x)=1$. Then there is a sequence $(G_n)$ of both
$\fs$- and $\gd$-subsets of $\om$ such that the function
$$
\ol{f}(x):=\frac{1}{2}\sum_{n=0}^\infty\left(\frac{2}{3}\right)^{n+1}\chi_{G_n}(x)  ~\mathrm{(uniform~convergence)}
$$
is a (positive) extension of $f$ with $\sup_{x\in\om}\ol{f}(x)=1$.
\end{theorem}
\begin{proof}
Let $\om$ be a perfectly normal topological space and let $A$ be a $\gd$-subset of $\om$.
Let also $f\in\B_1^{bd}(A)$ be a function such that $0\mn f\mn1$.
Obviously $0\mn2f\mn2$ and $-1\mn2f-1\mn1$. By Theorem \ref{MR}, every function $g=2f-1$ has an extension $\ol{g}\in\B_1(\om)$ such that $-1\mn\ol{g}(x)\mn1$ for every $x\in\om$. Thus

$$
\ol{g}=2\ol{f}-1=\sum_{n=0}^\infty\left(\frac{2}{3}\right)^{n+1}\left(\frac{1}{2}-\chi_{H_n}\right)
$$
and
$$
2\ol{f}=\frac{1}{2}\sum_{n=0}^\infty\left(\frac{2}{3}\right)^{n+1}-
\sum_{n=0}^\infty\left(\frac{2}{3}\right)^{n+1}\chi_{H_n}+1.
$$
Since
$$
\sum_{n=0}^\infty\left(\frac{2}{3}\right)^{n+1}=2,
$$
$$
\ol{f}=1-\frac{1}{2}\sum_{n=0}^\infty\left(\frac{2}{3}\right)^{n+1}\chi_{H_n}=
\frac{1}{2}\sum_{n=0}^\infty\left(\frac{2}{3}\right)^{n+1}-\frac{1}{2}\sum_{n=0}^\infty\left(\frac{2}{3}\right)^{n+1}\chi_{H_n}.
$$
Thus
$$
\ol{f}=\frac{1}{2}\sum_{n=0}^\infty\left(\frac{2}{3}\right)^{n+1}(1-\chi_{H_n})
$$
and
$$
\ol{f}=\frac{1}{2}\sum_{n=0}^\infty\left(\frac{2}{3}\right)^{n+1}\chi_{G_n},
$$
where $G_n=H_n'$ for every $n\in\mN$. That completes the proof.
\end{proof}

\section{Extensions of Baire-alpha functions}\label{EXT2}

Throughout this section the following abbreviations for some classes of subsets of $\om$ will be used:
\begin{description}
  \item[] $\prod^0_1(\om)\ \left[\sum^0_1(\om)\right]$ - closed [open] subsets of $\om$;
  \item[] $\prod_\al^0(\om)\ \left[\sum_\al^0(\om)\right]$ - the multiplicative [additive]
          class $\al$ of Borel subsets of $\om$, $0<\al<\omega_1$;
  \item[] $\Delta^0_\al(\om)$ - the family of Borel subsets of $\om$ of \emph{ambiguous class}
          $\al$ is defined to be $\prod_\al^0(\om)\cap\sum_\al^0(\om)$, $0<\al<\omega_1$;
\end{description}

The Extension Theorem for Baire-alpha functions is a consequence of Theorem \ref{MR} (for Baire-one functions), the general form of the Separation Theorem
(see \cite[Theorem 3.6.11]{Srivastava}) and the following general form of the Lebesgue-Hausdorff result.

\begin{theorem}(Lebesgue-Hausdorff, \cite[p.393]{Kuratowski2})\label{Leb_Haus}
Let $\om$ be a perfectly normal topological space and let a function $f\cn\om\to\mR$ be bounded.
For each ordinal $\al<\omega_1$, $f\in\B_\al(\om)$ if and only if $f^{-1}(F)\in\prod_\al^0(\om)$ for every closed subset $F$ of $\mR$.
\end{theorem}

We are now ready to present our second main result.

\begin{theorem}\label{MRALFA}
Let $\om$ be a perfectly normal topological space and let $A\subset\om$ be a Borel subset of multiplicative class $\alpha<\omega_1$.
Every function $f\in\B_\alpha(A)$ has an extension $\ol{f}\in\B_\alpha(\om)$ such that
$\|f\|_\infty=\|\ol{f}\|_\infty$.

More exactly, if $\sup_{x\in A}\abs{f(x)}=c<\infty$, then there exists a sequence
$\left(H_n\right)$ of Borel ambiguous class $\al$ subsets of $\om$, such that the function
$$
\ol{f}(x):=c\cdot\sum_{n=0}^\infty\left(\frac{2}{3}\right)^{n+1}\left(\frac{1}{2}-\chi_{H_n}(x)\right)
$$
extends $f$ with $\sup_{x\in\om}\abs{\ol{f}(x)}=c$.
\end{theorem}

\begin{proof}
As in the Tietze theorem, it is enough to consider the case $c=1$.
Put $M=f^{-1}\left(\left[-1,-\frac{1}{3}\right]\right)$ and
$N=f^{-1}\left(\left[\frac{1}{3},1\right]\right)$. By the Lebesgue-Hausdorff
theorem, both the (disjoint) sets $M$ and $N$ are of Baire, hence Borel,
multiplicative class $\al$ subsets of $A$. Since $\prod_\al^0(A)=A\cap\prod_\al^0(\om)$
and $A\in\prod_\al^0(\om)$, we have $M,N\subset\prod_\al^0(\om)$. By the
Separation Theorem, there is a set $H_0\in\Delta^0_\al(\om)$
such that $M\subset H_0$ and $N\cap H_0=\emptyset$. By the Lebesgue-Hausdorff
theorem again, the characteristic function $\chi_{H_0}$ belongs to $\B_\al(\om)$
(cf. \cite[Lemma 1.4]{ShatZaf}) and separates the sets $M$ and $N$.
Put $g_0(x)=\frac{2}{3}\left(\frac{1}{2}-\chi_{H_0}(x)\right)$.
Hence $g_0\in\B_\al(\om)$,
\begin{equation}\label{1}
|g_0(x)|=\frac{1}{3}\textrm{ for all }x\in\om\textrm{ and}
\end{equation}
\begin{equation}\label{2}
|f(x)-g_0(x)|\mn\frac{2}{3}\textrm{ for all }x\in A.
\end{equation}
Now we mimic the classical proof of the Tietze theorem \cite[p. 69]{Eng}.
Repeating the above arguments we obtain a sequence $\left(H_n\right)_{n=0}^\infty$
of sets of Borel ambiguous class $\al$ subsets of $\om$ such that, for the functions
$g_n\in\B_\al(\om)$, $n=0,1,\ldots$, of the form
$g_n(x):=\left(\frac{2}{3}\right)^{n+1}\left(\frac{1}{2}-\chi_{H_n}(x)\right)$
we have
\begin{equation}\label{1'}
|g_n(x)|=\frac{1}{2}\left(\frac{2}{3}\right)^{n+1}\textrm{ for all }x\in\om\textrm{ and}
\end{equation}
\begin{equation}\label{2'}
|f(x)-(g_0(x)+\ldots+g_n(x))|\mn \left(\frac{2}{3}\right)^{n+1}\textrm{ for all }x\in A.
\end{equation}
By (\ref{1'}) and \cite[Theorem 7, p. 421]{KurMost}, the series $\sum_{n=0}^\infty g_n$
converges uniformly on $\om$ to a function $g\in\B_\al(\om)$. By (\ref{2'}) we obtain
$g_{\restr A}=f$ and, by (\ref{1'}) again, $\sup_{x\in\om}|g(x)|=1$.
Hence $\ol{f}:=g$ is the required extension of $f$.
\end{proof}

It also has a version for non-negative functions.

\begin{theorem}\label{PMRALFA}
Let $\om$ be a perfectly normal topological space and let $A\subset\om$ be a set of multiplicative Borel class $\alpha<\omega_1$. Moreover, let $0\mn f\in B_\al^{bd}(A)$ with $\sup_{x\in A}f(x)=1$.
Then there is a sequence $(G_n)$ of Borel ambiguous class $\al$ subsets of $\om$ such that the function
$$
\ol{f}(x):=\frac{1}{2}\sum_{n=0}^\infty\left(\frac{2}{3}\right)^{n+1}\chi_{G_n}(x)  ~\mathrm{(uniform~convergence)}
$$
is a (positive) extension of $f$ with $\sup_{x\in\om}\ol{f}(x)=1$.
\end{theorem}

\section{The case of Polish spaces}
We start this section with a theorem of Kuratowski (see \cite[22D]{Kechris}) 
concerning topology refinements for Polish spaces.

\begin{theorem}[Kuratowski]\label{Kur}
Let $(\om,\tau)$ be a Polish space and $A_n\subseteq\om$ be $\Delta^0_\al(\om,\tau)$. 
Then there is a Polish topology $\tau'\supseteq\tau$ such that 
$\tau'\subseteq\sum_{\al}^{0}(\om,\tau)$ and $A_n\in\Delta^0_1(\om,\tau')$ for all $n$.  
\end{theorem}

It turns out that if $\om$ is Polish, our main conclusions can be reached with a few line 
argument using the above theorem.

\begin{theorem}\label{pol1}
  Let $(\om,\tau)$ be a Polish topological space and $A$ be a $\gd$-subset of $\om$.
  Then every function $f\in\B_1(A)$ has an extension $\ol{f}\in\B_1(\om)$ such that
  $\|f\|_\infty=\|\ol{f}\|_\infty$.
\end{theorem}

\begin{proof}
Let $f$ be a Baire-one function on a $\gd$-subset $A$ of $(\om,\tau)$. By Theorem
\ref{Kur} it is possible to refine the topology $\tau$ into $\tau'$ so that 
$f\cn(\om,\tau')\to\mR$ becomes continuous, $A\subset(\om,\tau')$ closed and every
set from $\tau'$ is $\fs$ in $\tau$ (indeed, taking $\left(U_k\right)$ as a countable
basis of $\mR$, this can be done by decomposing $\left(f^{-1}(U_k^c)\cap A\right)^c$
into the union of countably many closed sets which - by Theorem \ref{Kur} - are both
closed and open in $\tau'$). Now we extend continuous $f$ defined on a closed $A$
with the usual Tietze theorem to get $\ol{f}$, which is continuous on $(\om,\tau')$.
In the original domain this extension is Baire-one and has the desired properties.
\end{proof}

With an analogous argumentation we get the extension result for all Baire classes.

\begin{theorem}\label{polalfa}
Let $(\om,\tau)$ be a Polish topological space and let $A\subset\om$ be a Borel subset 
of multiplicative class $\alpha<\omega_1$. Every function $f\in\B_\alpha(A)$ has an extension 
$\ol{f}\in\B_\alpha(\om)$ such that $\|f\|_\infty=\|\ol{f}\|_\infty$.
\end{theorem}

\subsection*{Open Problem}
Is the Kuratowski's refinement technique true in perfectly normal topological spaces?
If yes, it is possible to formulate the above Theorems \ref{pol1} and \ref{polalfa}
in full generality.

\subsection*{Acknowledgements}
\begin{itemize}
\item The author would like to express his gratitude to Professor Marek Wójtowicz for his help in preparation of this manuscript.
  \item The author wishes to express his gratitude to the Reviewer, who kindly noticed
the possibility of significantly simplyfying the proof in the case of the domain
being a Polish space. Section $5$ is based on this observation.
\end{itemize}

\end{document}